\documentclass[oneside,reqno, twoside, 11pt]{amsart}
\usepackage[pdftex]{graphicx}
\usepackage{color}
\usepackage[T1]{fontenc}
\usepackage{lmodern}
\usepackage[english]{babel}
\usepackage{mathrsfs}  

\usepackage{upgreek}

\usepackage{amsfonts}
\usepackage{extpfeil}
\usepackage{verbatim}

\usepackage{BOONDOX-cal}

\usepackage{microtype}
\usepackage[new]{old-arrows}
\usepackage{amsmath}
\usepackage[all, cmtip]{xy}
\usepackage{tikz-cd}

\usepackage{amsthm}
\usepackage{mathtools}

\usepackage{bbm}
\usepackage{paralist}
\usepackage[T1]{fontenc}
\usepackage[colorlinks,citecolor=magenta,pagebackref=true,urlcolor=magenta, bookmarksdepth=3]{hyperref}
\usepackage[nodisplayskipstretch]{setspace}
\usepackage{breqn}

\usepackage{nicefrac}

\usepackage[font=small,labelfont=bf]{caption}
\usepackage[labelformat=simple]{subcaption}

\usepackage{epsfig}

\DeclareFontFamily{OT1}{pzc}{}
\DeclareFontShape{OT1}{pzc}{m}{it}{<-> s * [1.10] pzcmi7t}{}
\DeclareMathAlphabet{\mathpzc}{OT1}{pzc}{m}{it}

\makeatletter
\newtheorem*{rep@theorem}{\rep@title}
\newcommand{\newreptheorem}[2]{%
	\newenvironment{rep#1}[1]{%
		\def\rep@title{#2~\ref{##1}}%
		\begin{rep@theorem}}%
		{\end{rep@theorem}}}

\makeatother

\theoremstyle{plain}
\newtheorem*{thm*}{Theorem}
\newtheorem{thm}{Theorem}[section]

\newtheorem{cor}[thm]{Corollary}
\newtheorem{lem}[thm]{Lemma}
\newtheorem*{lem*}{Lemma}

\newtheorem{conj}[thm]{Conjecture}

\theoremstyle{definition}

 \newcommand{\Z}{\mathbb{Z}}
 \newcommand{\R}{\mathbb{R}}

\newcommand{\tet}{\vartheta}

\newcommand{\sbseq}{\subseteq}
\newcommand{\spseq}{\supseteq}

\newcommand{\vanish}[1]{}

\def\im{\mathrm{im}}

\def\sbs\subset
\def\sbseq{\subseteq}

\def\langle{\left<}
\def\rangle{\right>}

\def\({\left(}
\def\){\right)}
\def\no={\,{\,|\!\!\!\!\!=\,\,}}

\def\no={\,{\,|\!\!\!\!\!=\,\,}}
\def\sbseq{\subseteq}

\def\sbseq{\subseteq}
\def\sbs\subset
\def\spseq{\supseteq}

\newcommand{\xqedhere}[2]{%
	\rlap{\hbox to#1{\hfil\llap{\ensuremath{#2}}}}}

\newcommand\Defn[1]{\textbf{#1}}
\newcommand{\cm}[1]{}

\newcommand\mbf[1]{\mathbf{#1}}

\newcommand\mr[1]{\mathrm{#1}}

\newcommand{\fld}{\mathbbm{k}}

\renewcommand\deg{\mr{deg}}
\newcommand{\bigslant}[2]{{\raisebox{.3em}{$#1$} \Big/ \raisebox{-.3em}{$#2$}}}
\renewcommand\emptyset{\varnothing}
\newcommand\x{\mathbf{x}}

\DeclareMathOperator{\susp}{susp}
\DeclareMathOperator{\cone}{cone}

\title{Beyond positivity in Ehrhart Theory}
\author[Karim Adiprasito]{Karim Alexander Adiprasito}
\address{{Karim Adiprasito \ \emph{and}\ Vasiliki Petrotou \ \emph{and}\  Johanna Steinmeyer}, Einstein Institute of Mathematics, Hebrew University of Jerusalem, 91904 Jerusalem, Israel \emph{and} Department of Mathematical Sciences, University of Copenhagen, 2100 Copenhagen, Denmark}
\email{adiprasito@math.huji.ac.il \emph{and} v.petrotou@uoi.gr \emph{and} johanna.steinmeyer@math.huji.ac.il}

\author[Stavros Papadakis]{Stavros Argyrios Papadakis}
\address{{Stavros Papadakis}, Department of Mathematics, University of Ioannina, Ioannina, 45110, Greece}
\email{spapadak@uoi.gr}
\author[Vasiliki Petrotou]{Vasiliki Petrotou}

\author[Johanna Steinmeyer]{Johanna Kristina Steinmeyer}

\date{October 19, 2022}

\subjclass[2010]{}

\begin{document}
	
\begin{abstract}
We study semigroup algebras arising from lattice polytopes, compute their volume polynomials (particularizing work of Hochster), and establish strong Lefschetz properties (generalizing work of the first three authors). This resolves several conjectures concerning unimodality properties of the $h^\ast$-polynomial of lattice polytopes arising within Ehrhart theory.

\end{abstract}
	
	\maketitle
	
\newcommand{\AR}{\mathcal{A}}
\newcommand{\IR}{\mathcal{I}}
\newcommand{\JR}{\mathcal{J}}
\newcommand{\BR}{\mathcal{B}}
\newcommand{\CR}{\mathcal{C}}
	\newcommand{\Mu}{M}
	\newcommand{\Soc}{\mathcal{S}\hspace{-1mm}\mathcal{o}\hspace{-1mm}\mathcal{c}}
	\newcommand{\Socl}{{\Soc^\circ}}

\section{Lattice polytopes and semigroup algebras}

The main player of this paper is a convex polytope $P$ all whose vertices lie in a lattice $\Lambda$. Ehrhart theory analyzes the generating function
\[\mathrm{Ehr}_P(t)\ :=\ \sum_{i=0}^\infty \#\{iP\cap \Lambda\} t^i.\]

We can then write 
\[\mathrm{Ehr}_P(t)\ =\ \frac{h^\ast_P(t)}{(1-t)^{d+1}},\]
where $d$ is the dimension of $P$ and $\ h^\ast_P(t)=h_0^\ast+h_1^\ast t+\ldots+h_d^\ast t^d\ $\ is a polynomial of degree at most $d$.

To interpret this algebraically, we recall the following well-known construction: 
Embed the polytope in $\R^{d+1}$ at height $1$ and consider the cone over it, 
\[\cone(P)\coloneqq \R_{\geq 0} (P \times \{1\}).\] 
As such, it generates a semigroup algebra 
\[\fld[^*P]\coloneqq \fld^*[\cone(P)\cap\Z^{d+1}],\] 
graded by the last coordinate. Here $\fld$ is any field, though we  generally assume the field to be infinite to ensure the existence of an Artinian reduction. In this case $\fld[P]$ is Cohen-Macaulay \cite{Hochster} with Hilbert series $\mathrm{Ehr}_P(t)$. For a choice of linear system of parameters $\theta_0,\ldots,\theta_d\in\AR^1(P)$, the Artinian reduction
\[\AR^*(P)\coloneqq \fld^*[P]/\langle \theta_0,\ldots,\theta_d\rangle\] 
has $\dim\AR^k(P)=h^\ast_k$.  It follows that the $h^{\ast}_k$ are nonnegative.

It has been a central question in the theory of lattice polytopes to determine additional properties for these coefficients.
In particular, Hibi and Ohsugi conjectured that under two special conditions, the coefficients form a unimodal sequence \cite{OH}, see also \cite{Braun, SL}.

The first of these conditions is that $P$ has the \Defn{integer decomposition property}, short \Defn{IDP}: every lattice point of $\cone(P)$ is a nonnegative integral combination of lattice points in $P \times \{1\}$, or equivalently that $\fld^*[P]$ is generated in degree one.   

The second property is the \Defn{reflexive} property: there is a lattice point $p$ in $\Z^{d}\times\{1\}$ such that
\[\cone^\circ(P)\cap\Z^{d+1} =\ p+\cone(P)\cap\Z^{d+1} ,\]
where $\cone^\circ(P)$ is the interior of $\cone(P)$. This is equivalent to $\fld[P]$ being algebraically Gorenstein (that is, a Poincar\'e duality algebra after any Artinian reduction) with socle degree $d$ as well as to $h_k^{\ast}=h_{d-k}^{\ast}$ for all $k\leq \nicefrac{d}{2}$. Numerically on the level of the $h^*$-vector, the restriction to reflexive polytopes rather than all Gorenstein polytopes is without loss of generality. Bruns and Herzog \cite{BR07} showed that for every Gorenstein polytope there is a reflexive polytope with the same $h^*$-vector.

We resolve the following conjecture of Hibi and Ohsugi.

\begin{conj}[Hibi-Ohsugi]\label{conjecture}
	For any IDP reflexive lattice polytope $P\subset\R^d$, the coefficients of the $h^{\ast}$-polynomial are unimodal:
	\[ h_0^{\ast}\leq h_1^{\ast}\leq\ldots\leq h_{\lfloor \nicefrac{d}{2}\rfloor}^{\ast} = h_{\lceil \nicefrac{d}{2}\rceil}^{\ast} \geq \ldots \geq h_d^{\ast}\]
\end{conj}

This is the updated form of a conjecture of Hibi \cite{Hibi92}, after Musta{\c t}a and Payne gave an example showing the necessity of the IDP assumption \cite{MP}. These conjectures in turn go back to a more general one of Stanley \cite{Stanley89}, who proposed that the unimodality may hold for a general Gorenstein standard graded integral domain. 

In fact, we shall prove statements that are more powerful than this, and apply to more general cases. For instance, in the case of polytopes that have only the IDP, we still obtain monotone decreasing coefficients in the second half, i.e.
\[ h^{\ast}_{\lfloor \nicefrac{d}{2}\rfloor}\geq \ldots \geq h_d^{\ast}. \]
Moreover, we refine this further by studying also the \emph{local} $h^{\ast}$-polynomial of a polytope, and obtain unimodality of this refined invariant for any IDP lattice polytope.

We prove Conjecture \ref{conjecture} by proving a Lefschetz property for a generic Artinian reduction of the associated semigroup algebra. 

\begin{thm}\label{thm:lef}
	If $P$ is an IDP reflexive polytope, and the characteristic of $\fld$ is $2$ or $0$, then a generic Artinian reduction $\AR^*(P)$ of  $\widetilde{\fld}^*[P]$ has the Lefschetz property, i.e. there is a linear element $\ell\in\AR^1(P)$ such that for any $k\leq \nicefrac{d}{2}$, the map
	\[ \AR^k(P) \xrightarrow{\ \cdot \ell^{d-2k}\ }\ \AR^{d-k}(P)\]
	is an isomorphism.
\end{thm}

Here, \emph{generic} shall mean that the Artinian reduction is taken by linear forms 
\[\theta_0,\ldots,\theta_d,\quad \text{where} \quad \theta_i = \sum_{p\in P\cap\Lambda} \theta_{i,p},\] with transcedentally independent coefficients $\theta_{i,p}$, necessitating passing to a transcendental field extension $\widetilde{\fld}$ of $\fld$. 

This is a rather extreme choice of linear system of parameters, necessitated by the proof via anisotropy. We want to emphasize the importance of the choice of l.s.o.p., as it makes a crucial difference for the ring we end up working with. 
For a specific, and somewhat canonical, choice of l.s.o.p., $\AR^*(P)$ is isomorphic to the orbifold Chow ring of the associated toric Deligne-Mumford stack \cite{BCS}. In contrast to this choice, we make of a generic Artinian reduction here. As graded vector spaces, the results are isomorphic and the inequalities on the dimensions of the graded pieces remain true. Yet, as observed in \cite{AHL}, whether an Artinian reduction admits a Lefschetz element depends on the Artinian reduction, and not only on $\fld^*[P]$. That such ``bad'' Artinian reductions exist in the present context was observed in \cite{BD16}, giving an example of an IDP reflexive simplex and an Artinian reduction of the associated semigroup algebra which does not even admit a weak Lefschetz element. The linear system of parameters chosen there is however not the canonical system for the orbifold Chow ring, and whether that ring has the Lefschetz property remains open.

As the composition of individual multiplications with $\ell$, the Lefschetz isomorphism gives us an injection in the first half and a surjection in the second half, and thus the desired inequalities on the dimensions of the graded pieces:

\begin{thm}
The $h^\ast$-polynomial of a reflexive IDP lattice polytope has a unimodal sequence coefficients. Stronger, we get that \ $(h_{i+1}^{\ast}-h_{i}^{\ast})_{0\leq  i\leq \nicefrac{d}{2}}$ \ is an $M$-vector.
\end{thm}

Here the $M$-vector property too immediately follows from the Lefschetz property.
The vector of differences $\max\{h^\ast_i- h^\ast_{i-1},0\}$  is the Hilbert vector of a standard graded algebra, namely 
\[\bigslant{\mathcal{A}^\ast(P)}{\ell \mathcal{A}^\ast(P)}.\]

We follow the recent breakthroughs of Adiprasito \cite{AHL}, Papadakis and Petrotou \cite{PP, APP}, though our work requires a critical new ingredient. This is due to the fact that the proofs employed previously for Stanley-Reisner rings made use of explicitly combinatorial techniques to reach the desired goal, and the algebra we investigate here is not immediately as governed by a combinatorial struture as the previous one, being cut out by binomial ideals rather than monomials. And so while Hochster was the one to gift us with understanding of the canonical module, he was unable to determine a natural normalization of the fundamental class. We give one here, compatible with the case of Stanley-Reisner rings. From this, we then get the key identity, reminiscent of Parseval's identity in Fourier analysis. 

The paper is organized as follows. We start in Section~\ref{sec:lattice} by generalizing the setup in order to state our main theorem and deduce the individual numerical corollaries within Erhart theory. For sake of completeness, we then recall in Section~\ref{sec:ani} the necessary parts of the machinery of \cite{AHL,PP,APP} in order to prove the Lefschetz statements by way of anisotropy. Following this thorough setup, we give our new contributions to the theory. Section~\ref{sec:deg} contains the normalization of the fundamental class and an auxiliary differential identity, while Section~\ref{sec:parseval} contains the key identity of Parseval type. We finish with a discussion of remaining open questions in Section~\ref{sec:discussion}.

\section{Poincar\'e, Gorenstein and the Lefschetz properties}\label{sec:lattice}

For $P$ reflexive, the semigroup algebra $\fld^*[P]$ is Gorenstein of Krull dimension equal to the dimension of the polytope plus one \cite{Bruns-Herzog}. After an Artinian reduction using a linear system of parameters of length equal to the Krull dimension, we arrive at a Poincar\'e duality algebra of socle degree $d$. 

For general IDP polytopes, the situation is a little more delicate. One can force Poincar\'e duality however, using the usual trick: we allow for relative objects. 
The $\fld^*[P]$-module defined by
\[ \fld^*[P,\partial P]\coloneqq\fld^*[(\cone(P)\setminus\partial \cone(P))\cap\Lambda'] \]
is the canonical module of the Cohen-Macaulay ring $\fld^*[P]$ \cite{Hochster}.
After Artinian reduction, we are left with a perfect bilinear pairing
\[\AR^k(P)\ \times\ \AR^{d+1-k}(P,\partial P)\ \longrightarrow\ \AR^{d+1}(P,\partial P)\ \cong\ \fld.\]

\begin{thm}\label{thm:rellef}
If $P$ is an IDP polytope of dimension $d$, and the characteristic of $\fld$ is $2$ or $0$, then a generic Artinian reduction $\AR^*(P)$ of  $\widetilde{\fld}^*[P]$ has the relative Lefschetz property, i.e. there exists a linear element $\ell\in\AR^1(P)$ such that for all $k\leq \nicefrac{d+1}{2}$,
\[\AR^k(P,\partial P)\ \xrightarrow{\ \cdot \ell^{d+1-2k}\ }\ \AR^{d+1-k}(P)\]
is an isomorphism. 
Moreover, the system of parameters can be chosen such that the restriction to any face forms a linear system of parameters when restricted to that face (meaning that its length does not exceed the depth of the module on that face.)
\end{thm}
The final part here is important for the refined statement about local $h^\ast$-vectors, and we call it the \emph{perverted} linear system, because it is the most generic linear system agreeing with the intersection cohomology of the associated orbifold.

Note that Theorem~\ref{thm:rellef} implies Theorem~\ref{thm:lef}. For $P$ reflexive, we have an isomorphism of semigroup algebras $\fld^k[P]\xrightarrow{\cdot \x_p}\fld^{k+1}[P,\partial P]$ which passes to the Artinian reduction. For $k\leq\nicefrac{d}{2}$, we then get the Lefschetz isomorphism as the composition
\[\AR^k(P)\xrightarrow{\cdot \x_p}\AR^{k+1}(P,\partial P) \xrightarrow{\ \cdot \ell^{d-2k-1}} \AR^{d-k}(P).\]

For general IDP polytopes, Theorem~\ref{thm:rellef} still gives surjections \[\AR^{d-k}(P)\xrightarrow{\cdot\ell}\AR^{d+1-k}(P)\quad \text{and}\quad \AR^k(P)\xrightarrow{\cdot\ell^{d+1-2k}}\AR^{d+1-k}(P)\] for $k\leq \lfloor \nicefrac{d}{2}\rfloor$ and thus

\begin{cor}
	The $h^{\ast}$-polynomial of an IDP lattice polytope of dimension $d$ has monotone decreasing coefficients in the second half, i.e.
	\[ h^{\ast}_{\lfloor \nicefrac{d}{2}\rfloor}\geq \ldots \geq h_d^{\ast}. \]
	Moreover, for all $k \le \nicefrac{d+1}{2}$, we have
	\[h^{\ast}_{k}\ \ge\ h^{\ast}_{d+1-k}.\]
	In particular, for any IDP lattice polytope, Stapledon's $a$-polynomial has unimodal coefficients \cite{Stapledon}.
\end{cor}

The last part also follows from Theorem~\ref{thm:lefsphere}, by observing that the $a$-polynomial corresponds exactly to the $h^*$-polynomial of $\partial P$ as a lattice complex.

In addition, Theorem \ref{thm:rellef} has interesting consequences if we know at what height the interior of $\cone^\circ(P)$ is generated. Suppose all minimal elements of $\cone^\circ(P)$ are of heights at most~$j$. Then, for $k\leq \lceil\frac{d-j}{2}\rceil$, every nontrivial $u$ in $\mathcal{A}^k(P)$ multiplies with one of these elements, say $\x_\alpha$. 
 Hence in the composition
\[\mathcal{A}^k(P) \ \xhookrightarrow{\ \cdot \x_\alpha\ }  \ \mathcal{A}^{k+j}(P,\partial P) \xrightarrow{\ \cdot \ell^{d+1-2k-2j}\ } \mathcal{A}^{d+1-k-j}(P),\]
the element $u$ is nontrivial. Hence $ \ell^{d+1-2k-2j} u$ is nontrivial.

\begin{cor}
The $h^\ast$-polynomial of an IDP lattice polytope of dimension $d$ with $\cone^\circ(P)$ generated at height $\leq j$ has monotone increasing coefficients in the initial part, i.e. 
\[h^{\ast}_0\leq\ldots\leq h^{\ast}_{\lceil\frac{d-j}{2}\rceil}.\]
Moreover, for $k\leq \lceil\frac{d-j}{2}\rceil$ we have 
\[h_k^*\leq h^*_{d+1-j-k}.\]
\end{cor}

For the finer statements regarding the local $h^{\ast}$-vector, we need a more powerful algebraic statement. To this end, notice that given any abstract polytopal complex $\Sigma$ consisting of lattice polytopes (short lattice complex), with the property that the lattices coincide in common intersections, we obtain an analogous Cohen-Macaulay ring $\fld^*[\Sigma]$ \cite{BBR}. In analogy with the proofs of the g-theorem of \cite{AHK,PP,APP}, we have

\begin{thm}\label{thm:lefsphere}
If $\Sigma$ is an IDP lattice sphere or ball (that is, a polyhedral sphere or ball with lattice structure in which every element has the IDP property) of dimension $d$, and the characteristic of $\fld$ is $2$ or $0$, then a generic Artinian reduction $\AR^*(\Sigma)$ of  $\widetilde{\fld}^*[\Sigma]$ has the Lefschetz property
\[\AR^k(\Sigma,\partial \Sigma)\ \xrightarrow{\ \cdot \ell^{d-2k}\ }\ \AR^{d+1-k}(\Sigma).\]
\end{thm}

Similar results hold for manifolds and cycles (again in direct analogy to \cite{APP}) but seem less immediately relevant here. 

For the case that $\Sigma$ is a sphere, we of course have $\AR^*(\Sigma,\partial)=\AR^*(\Sigma)$, letting us consider the standard graded algebra $\bigslant{\mathcal{A}^\ast(\Sigma)}{\ell \mathcal{A}^\ast(\Sigma)}$. We denote the dimensions of the graded pieces by $g_k^*$.  

For any polytope $P$ with the integer decomposition property, we can now combine the relative property for $P$ with the Lefschetz property for the boundary sphere $\partial P$. 

Consider to that effect the exact sequence \[\widetilde{\fld}^*[P,\partial P] \ \longrightarrow \ \widetilde{\fld}^*[P] \ \longrightarrow \ \widetilde{\fld}^*[\partial P] \ \longrightarrow \  0.\]

In order to pass to a statement on the Artinian reductions, we take a perverted linear system of parameters for $P$, ensuring that the restriction to $\partial P$ is again a proper linear system of parameters. With the last element $\theta_d$ acting as Lefschetz element for $\AR^*(\partial P)$, we get the short exact sequence

\[ 0 \ \longrightarrow \ \im\Big(\mathcal{A}^*(P,\partial P) \to \mathcal{A}^*(P)\Big) \  \longrightarrow \ \AR^*(P) \ \longrightarrow \ \bigslant{\mathcal{A}^\ast(\partial P)}{\theta_d \mathcal{A}^\ast(\partial P)} \ \longrightarrow \ 0.\]

We want to use this observation to analyze the local $h^*$-polynomial $\ell^*(t)$ (also called the Box polynomial or $\tilde{S}$-polynomial) considered by Betke-McMullen, Stanley, and others \cite{BM,Stanleysub,BM03,KatzStapledon}. 
From Stanley \cite{Stanleysub}[Theorem~7.8, Example~7.13], we get for any Cohen-Macaulay lattice complex $\Sigma$ that
\[ h_\Sigma^*(t) \ = \sum_{\emptyset\leq F\in\Sigma} \ell_F^*(t)\cdot h_{[F,\Sigma]}(t) \  = \sum_{\emptyset\leq F\in\Sigma} \ell_F^*(t)\cdot g_{[F,\Sigma)}(t), \]
where the sum goes over all faces of $\Sigma$ and $h_{[F,\Sigma]}(t), g_{[F,\Sigma)}(t)$ are the toric $h$- and $g$-polynomial respectively, of the posets $[F,\Sigma]$ and $[F,\Sigma)$, given as an interval in the face lattice of $\Sigma$.

From this we deduce
\begin{align*}
	\ell^*_P(t) \ &= \ h_P^*(t) \ - \sum_{\emptyset\leq F\in\partial P} \ell_F^*(t)\cdot h_{[F,P]}(t) \\ &= \ h_P^*(t) \ - \sum_{\emptyset\leq F \in \partial P} \ell_F^*(t)\cdot g_{[F,\partial P]}(t) \\
	&= \ h^*_P(t) - g^*_{\partial P}(t), 
\end{align*}
implying that the dimensions of the graded pieces of $\im\Big(\mathcal{A}^*(P,\partial P)\to \mathcal{A}^*(P)\Big)$ are exactly $\ell^*_k$.

The relative Lefschetz map now restricts to $\im\Big(\mathcal{A}^*(P,\partial P)\to \mathcal{A}^*(P)\Big)$ and thus

\begin{thm}
The local $h^\ast$-vector $\ell^{\ast}$ of any IDP polytope is unimodal.
\end{thm}

\section{Anisotropy}\label{sec:ani}

To deduce the Lefschetz property, we employ a reduction found in \cite{AHL}, see also \cite{APP}: It is enough to demonstrate a nondegeneracy property of the Poincar\'e pairing at certain ideals. The following theorem then finally implies the Lefschetz theorems.

\begin{thm}\label{thm:ani}
	If $P$ is an IDP $d$-polytope, and the characteristic of $\fld$ is $2$, then a generic Artinian reduction $\AR^*(P, \partial P)$ of  $\widetilde{\fld}^*[P,\partial P]$ has the anisotropy property, that is,
	for every nontrivial $u\in\AR^{k}(P,\partial P)$ of degree $k\le \nicefrac{d}{2}$, we have $u^2\ \neq\ 0$.

	The analogous result applies to lattice discs and spheres.
\end{thm}

Most of the remainder of this paper is devoted to proving this theorem. Before we do that, however, let us note that from here, the Lefschetz property is a cakewalk. For sake of completeness, we remind the reader of the derivation of the Lefschetz property from anisotropy, based on \cite{AHL}. 

Consider an IDP lattice sphere of dimension $d-1$. We say that $\AR^\ast(\Sigma)$, with socle degree $d$, satisfies the \Defn{Hall-Laman relations} in degree $k\le \frac{d}{2}$ and for ideal $\mathcal{I}^\ast\subset \AR^\ast(\Sigma)$ if there exists an $\ell$ in $\AR^1(\Sigma)$, such that the pairing
	\begin{equation}\label{eq:sl}
	\begin{array}{ccccc}
	\mathcal{I}^k& \times &\mathcal{I}^k & \longrightarrow &\ \mathcal{I}^d\cong \fld \\
	a	&		& b& {{\xmapsto{\ \ \ \ }}} &\ \mr{deg}(ab\ell^{d-2k})
	\end{array}
	\end{equation}
	is nondegenerate.
	
Note that that the Hall-Laman relations for an ideal $\AR^\ast(\Delta,\partial \Delta)$ corresponding to a disk~$\Delta$ of dimension $d$ in $\Sigma$ are precisely the relative Lefschetz property of that disk; hence the Hall-Laman relations give us a way to discuss relative Lefschetz properties properly.
	
	To in turn prove the Hall-Laman relations, consider the suspension $\susp \Sigma$ of the lattice disc or sphere $\Sigma$, which is again IDP. Note that $\Sigma$ itself is a face of $\susp\Sigma$, so two points with one in each of the cones over $\Sigma$ do not lie in a common face. On the level of semigroup algebras, a suspension corresponds to the introduction of two variables which multiply to zero. Label the two vertices (corresponding to the newly introduced indeterminates) of the suspension $\mbf{n}$ and $\mbf{s}$ (for north and south). Let $\uppi$ denote the projection along $\mbf{n}$, and let $\tet$ denote the height over that projection.

	\begin{lem}[{\cite[Lemma 7.5]{AHL}}]\label{lem:midred}
Considering $\susp\Sigma$ realized in $\fld^{d+1}$, and $k< \frac{d}{2}$, the following two are equivalent:
		\begin{compactenum}[(1)]
					\item The Hall-Laman relations in degree $k+1$ for $\IR\cdot \x_{\mbf{n}}$ in $\AR^{k+1}(\susp \Sigma)$
					with respect to $\x_{\mbf{n}}$.									
			\item The Hall-Laman relations in degree $k$ for 
 $\IR$	in $\AR^k(\Sigma)$
with respect to $\tet$. 
		\end{compactenum}
	\end{lem}	

This reduces the Lefschetz property iteratively to a pairing property for $d+1=2k$ and such that
\begin{equation*}
\begin{array}{ccccc}
\mathcal{J}^k& \times &\mathcal{J}^k & \longrightarrow &\ \mathcal{J}^d\cong \fld \\
a	&		& b& {{\xmapsto{\ \ \ \ }}} &\ \mr{deg}(ab)
\end{array}
\end{equation*}
is nondegenerate on some $d$-dimensional lattice disc or sphere. But this is guaranteed for any $\mathcal{J}$ by the anisotropy property of Theorem~\ref{thm:ani}. The result for spheres and discs follows simply by linearity, as each monomial lies in a polytopal face.

We are left with the task of proving Theorem~\ref{thm:ani} in the case when $2k$ equals the dimension $d$ of the polytope plus one, so that 
\[\AR^{2k}(P,\partial P)\ \cong\ \widetilde{\fld}.\]

Given the Lefschetz property over \emph{some} field of characteristic $2$, note that the Lefschetz property over \emph{any} field of characteristic $2$ or $0$ immediately follows: The injectivity of the multiplication map is a condition on the determinant of a linear map. If it does not vanish in characteristic $2$, it cannot vanish in characteristic $0$.

\section{Normalization of the integration map}\label{sec:deg}

We now want to prove anisotropy for a general IDP lattice polytope $P$ relative to its boundary, so that in a generic Artinian reduction, $\AR^{d+1}(P,\partial P)\cong\tilde{\fld}$ is the fundamental class. As this is a statement about elements of the fundamental class not vanishing, it serves us well to find a way to make this last isomorphism explicit. We denote the resulting map by $\deg: \AR^{d+1}(P,\partial P)\to\tilde{\fld}$. It is usually called the degree map, but to avoid confusion with the degree of a monomial, we shall instead call this identification the \emph{integration map} (alluding to the fact that what we are aiming to understand is actually the volume polynomial).

One curiosity of lattice polytopes is that even though we know the canonical module thanks to Hochster, we do not actually know of an identification of the top degree with the base field, that is, we do not have a canonical choice of integration map. The proof of Theorem~\ref{thm:ani} is reliant on understanding this integration. In the situation of classical algebraic geometry, there is a canonical such identification, which leads to a classical combinatorial formula in toric geometry \cite{Brion}. In our case, no such canonical identification seems to have been explored. We provide it here.

Stopping short of a full combinatorial formula, we contend ourselves with determining it uniquely. As this is an isomorphism between to copies of the same field, it suffices to give one nontrivial affine condition and prove consistency in order to determine the function.

If $P$ has a unimodular boundary facet $\tau$ in $\partial P$, then we obtain the desired normalization by matching the face ring picture and setting 
\begin{equation}\label{eq:normie}
1\ =\ \sum_{p\in (P\setminus\tau)\cap\Lambda} \deg(\x_p\x_\tau) \det(\Theta_{|\tau,p}) 
\end{equation}
where $\Theta=(\theta_{i,j})$ is the matrix of coefficients in the linear system of parameters. In general, consider a flag $(\tau_i)$ of faces of $P$ such that $\tau_d=P$ and such that $\tau_i$ is a facet of $\tau_{i+1}$. We say a set $\sigma=\{\sigma_0,\ldots,\sigma_d\}\subset(P\cap\Lambda)^{d+1}$ without repetitions is \Defn{coherent} with $(\tau_i)$ if it intersects $\tau_i$ in a set of cardinality $i+1$. We then normalize by setting
\begin{equation}\label{eq:norm2}
1\ =\ \sum_{\sigma \text{ coherent with } (\tau_i)} \deg(\x_\sigma) \det(\Theta_{|\sigma}) 
\end{equation}

\begin{lem}\label{lem:transfer}
The normalization is independent of the flag $(\tau_i)$ chosen.
\end{lem}

\begin{proof}
A direct computational proof is rather uninformative, but we can use an indirect argument without getting our fingers dirty that informs what is really happening. Let us assume first that $P$ has a boundary consisting of unimodular simplices. Consider then the cone over $P$ as a polytope; it follows that 
\[(v\ast \partial P)\ \cup\  P\]
is a lattice complex in the sense of \cite{BBR}, and that it defines a Poincar\'e duality algebra, meaning that the integration map on $P$ extends uniquely to a integration map on $(v\ast \partial P) \cup P$. For the simplices of $v\ast \partial P$ the integration map is uniquely determined and given by $\det^{-1}(\Theta_{|X})$ for a simplex $X$. The consistency follows.

For an arbitrary polytope $P$, we argue in a similar fashion. However, we can use the fact that the faces of $v\ast \partial P$ are decidedly more simplicial than $P$: they are cones over polytopes of codimension one. We can repeat that trick and obtain
double cones over arbitary polytopes as new faces, and so on; finally the facets are simplicial and we once again understand the integration map in terms of smooth toric varieties. This implies consistency of the integration map conditions.
\end{proof}

In other words, the normalization is compatible to the natural integration map in toric geometry, and is consistent with the same if we use gluing to smooth affine varieties. This means that the normalization is somewhat canonical; in fact, it also coincides with the normalization obtained by Kustin-Miller unprojection \cite{Papadakis}, a connection which we shall explore in an upcoming work.

This definition of the integration map as well as the anisotropy we wish to prove is dependent of our choice of $\theta_i$, so we will consider $\deg$ as a rational function in the variables $\theta_{i,j}$. This allows us to formulate the auxiliary lemma on the way to anisotropy.

\begin{lem}\label{lem:diffid}
	For $P$ a $(2k-1)$-dimensional polytope, the integration normalized, and $F=\{f_1,\ldots,f_k\}\subset (P\cap\Lambda)^k$ is without repetitions, in $\AR(P)$ over any infinite field of characteristic $2$, we have 
	\[\partial_F \deg(u^2)\ =\ \deg(u \cdot\x_F)^2.\]
\end{lem}

Here, $\partial_F$ is the differential operator obtained as the composition of the differential in directions 
$\theta_{1,f_1}, \theta_{2,f_1}, \theta_{3,f_2}, \theta_{4,f_2}, \theta_{5,f_3}, \dots$, applied to the integration map as a rational function. 

We postpone the proof of this lemma to the end of the next section, where we will derive it from the key identity of Parseval type. 

From this differential identity, anisotropy follows at once: every nonzero element~$u$ as in the lemma multiplies with \emph{some} $\x_F$ due to Poincar\'e duality. Hence $\partial_F\deg(u^2)$ is not identically zero. Thus for a choice of transcendentally independent coefficients, $\deg(u^2)\neq 0$, proving Theorem~\ref{thm:ani}.

\section{Parseval identities}\label{sec:parseval}

While the differential identity is similar to identities proven in the case of simplicial cycles in \cite{APP, PP}, the case of lattice polytopes is comparatively harder: we understand the integration map using the following nonhomogeneous equation that takes the form of an identity of the Parseval type. We consider polytopes of dimension $d$.

Assume    $\beta = (\beta_0,  \dots , \beta_{d} )   \in  (P\cap\Lambda)^{d+1}$. We set as usual
\[\x_\beta\ =\ \prod_{0 \leq i \leq d}  \x_{\beta_i}.\]

\begin{lem}[The Parseval identity (disguised)]\label{lem:parsevald}
For an IDP lattice $d$-polytope $P$, we have in $\AR^*(P,\partial P)$ over characteristic $2$
\begin{equation}\label{eq:id}
\deg (\x_\alpha) \   =\    \sum_{ \beta \in (P\cap \Lambda)^{d+1}}   \deg (\x_{\frac{\alpha+\beta}{2}})^2 \cdot \prod_{0 \leq i \leq d} \theta_{i,\beta_i} 
\end{equation}
\end{lem}

Here, $\deg (\x_{\frac{\alpha+\beta}{2}})$ is defined to be $\deg (\x_\gamma)$ if there is an  $\x_\gamma \in \fld[P]$ such that $\x_\alpha \x_\beta   = \x_\gamma^2$, and 0 otherwise.

This specializes to the following identity for $\alpha=2 \alpha'$, which explains the naming of this identity:

\begin{lem}[The Parseval identity (revealed)]\label{lem:parseval}
For an IDP lattice $d $-polytope $P$, in $\AR^*(P,\partial P)$ over characteristic $2$ and for $d+1$ even, we have
\begin{equation}
\deg (u^2) \   =\    \sum_{ \beta \in (P\cap \Lambda)^{d+1}}  \deg (u\cdot\x_{\frac{\beta}{2}})^2 \cdot \prod_{0 \leq i \leq d} \theta_{i,\beta_i} 
\end{equation}
\end{lem}

Here for a non-monomial $u^2$ we get the result by linearity over characteristic $2$:

\begin{align*}
	\deg(u^2) \ &= \ \deg\bigg(\big(\sum_a\lambda_a \x_a\big)^2\bigg) \ = \  \sum_a\lambda_a^2\ \deg(\x_a^2) \\
	&= \  \sum_a \lambda_a^2\sum_{ \beta \in (P\cap \Lambda)^{d+1}}  \deg (\x_a\cdot\x_{\frac{\beta}{2}})^2 \cdot \prod_{0 \leq i \leq d} \theta_{i,\beta_i} \\
	&= \ \sum_{ \beta \in (P\cap \Lambda)^{d+1}}  \deg (u\cdot\x_{\frac{\beta}{2}})^2 \cdot \prod_{0 \leq i \leq d} \theta_{i,\beta_i} 
\end{align*}

In order to prove the Parseval identity, we recall first a well-known identity for the integration map:

\begin{lem}[The balancing identity]
Consider any two elements $I,J \in (P \cap \Lambda)^{d}$, where at least one point of $I$ lies in the interior of $P$. Then we have
\[R_{I,J}\ :=\ \sum_{p \in P \cap \Lambda} \det(\Theta_{|J,p}) \deg(\x_{I}\x_p)\ =\ 0.\]
\end{lem}

This is a consequence (see \cite{PP}[11.1]) of the following, simpler identity:

\begin{lem}[The linear identity]
Consider any $0\leq s\leq d$, and $x_I$ a monomial of degree $d$ of $\widetilde{\fld}[P,\partial P]$. Then in $\AR(P,\partial P)$ we have 
$\sum_{p \in P \cap \Lambda} \theta_{s,p}\deg(\x_{I}\x_p)\ =\ 0 $.
\end{lem}

This last identity arises simply because we quotient by the linear elements \[\theta_j=\sum_{p \in P \cap \Lambda} \theta_{j,p} \x_p.\]

\begin{proof}[Proof of Lemma~\ref{lem:parseval}]
We establish a special case first, and argue by induction on dimension. Consider the case $d=1$.

Instead of proving the Lemma for every lattice line segment separately, we instead prove it for the halfline $P=\mathbb{Z}_{\ge 0}$, and restrict to segments afterwards; this is possible because normalization requires the existence of a flag of boundary faces only.

To simplify our work we index the monomials by the sum of their lattice points, as the integration map of a monomial only depends on that invariant by the arithmetic relations imposed by the lattice structure.

We then start by considering the normalization Identity~\eqref{eq:normie}, and multiply both sides by $\deg(\x_\alpha)$ to obtain
\[\deg(\x_\alpha)\ =\ \sum_{\sigma \text{ coherent with } \tau} \deg(\x_\alpha) \deg(\x_{|\sigma|}) \det(\Theta_{|\sigma})\]
where $|\sigma|$ denotes the sum of lattice points of $\sigma$. We denote the right side by $N_\alpha$.

There exists a natural action \[S_{d+1}  \times  (P \cap \Lambda)^{d+1} \  \longrightarrow\  (P \cap \Lambda)^{d+1}\]
of the symmetric group  $S_{d+1}$  on $d+1$ elements on $(P \cap \Lambda)^{d+1}$ 
and we set 
\[\omega_b\ =\  \sum_{\beta \in \mr{Orb}(b)} \prod_{0 \leq i \leq d}  \theta_{i,\beta_i}.\]
We can then rewrite Identity \eqref{eq:id} as
\[  \deg (\x_\alpha) \   =\    \sum_{ b \in \nicefrac{(P\cap \Lambda)^{d+1}}{S_{d+1}}}  \omega_b \cdot \deg (\x_{\frac{\alpha+\beta}{2}})^2.\]

Now, given $i, j \in \mathbb{Z}_{\ge 0}$, we set $\lambda_{i,j}$ to be $\delta_{i,j}\cdot\deg(\x_{\alpha+j-i})$. Here $\delta_{i,j}$ is $1$ if $\nicefrac{\alpha}{2}<i<j$ or $\nicefrac{\alpha}{2}-j<i<\nicefrac{\alpha}{2}$, and $0$ otherwise. Consider then
\begin{align*}
0 \ =\	& \sum_{i,j}\lambda_{i,j}R_{i,j}\\
=\	& \sum_{i,j}\lambda_{i,j}\sum_{t\in \mathbb{Z}_{\ge 0}}\det(\Theta_{|j,t})\deg(\x_{i+t})\\
	=\ & N_\alpha\ +\ \sum_{t<j\in \mathbb{Z}_{\ge 0}} \omega_{t,j}\cdot\left(\deg(\x_{\frac{\alpha+t+j}{2}})^2\ +\ \deg(\x_{\frac{\alpha}{2}+t})\deg(\x_{\frac{\alpha}{2}+j})\right).
\end{align*}

Applying the linear identities we obtain that
\[0\ =\ N_\alpha\ +\ \sum_{t\le j\in \mathbb{Z}_{\ge 0}} \omega_{t,j}\cdot\deg(\x_{\frac{\alpha+t+j}{2}})^2.\]
This proves the Parseval identity for $d=1$ for the halfline. Now we can restrict to any line segment $P'= \{t\in \mathbb{Z}_{\ge 0}: p\le o\}$ to obtain the general statement: here, we note that the coefficients $\lambda_{i,j}$ are chosen providently to vanish whenever $i$ is not in the interior of the given line segment.

To prove the Parseval identity for $d>1$, we consider again first the polyhedral cone 
$P=\mathbb{Z}_{\ge 0}^{[d]}$, $[d] =\{1,\dots,d\}$, where the initial flag is chosen to be.
\[\mathbb{Z}_{\ge 0}^{[0]}=\{0\}\ \subset\ \mathbb{Z}_{\ge 0}^{[1]}\ \subset\ \mathbb{Z}_{\ge 0}^{[2]}\cdots.\]
Define the lexicographic order so that the this flag is minimal.

For $I, J \in \mathbb{Z}_{\ge 0}^{[d]}$, we set $\lambda_{I,J}$ to be $\delta_{I,J}\cdot\deg(\x_{\alpha+J-I})$, where we choose $\delta_{I,J}$ inductively on dimension: given $j$ the lexicographically minimal element of $J$, set $\delta_{I,J}:= \delta_{I,J\setminus \{j\}}$ with respect to the pullback to the ideal of $\x_j$. This finishes the proof.
\end{proof}

\begin{proof}[Proof of Lemma~\ref{lem:diffid} using Lemma~\ref{lem:parseval}]  
For $P$ a $(2k-1)$-dimensional polytope, the integration map normalized and $F$ a $k$-set of lattice points in $P$ without repetitions, let $u\in \AR^{2k}(K_P)$. Consider $\deg(u^2)$ using \eqref{eq:id}. We have the revealed Parseval identity
\[\deg (u^2) \   =\  \sum_{ \beta \in (P\cap \Lambda)^{d+1}}  \deg (u\cdot\x_{\frac{\beta}{2}})^2 \cdot \prod_{0 \leq i \leq d} \theta_{i,\beta_i}  \]
Now differentiation by $F$ yields
\[ \partial_F\deg (u^2)\ =\ \sum_{ \beta \in (P\cap \Lambda)^{d+1}}  \deg (u\cdot\x_{\frac{\beta}{2}})^2  \cdot \partial_F \left( \prod_{i=0}^d \theta_{i,\beta_i}\right),\]
in characteristic two, as differentiating $\deg(u\cdot\x_{\frac{\beta}{2}})^2$ introduces a factor of 2. For all nonzero summands there exists a $b$ such that $2b=\beta$.
Now note that $\partial_F \big( \prod_{i=0}^d \theta_{i,\beta_i}\big)=\delta_{F,b}$, so we get the desired identity
\[\partial_F \deg(u^2)\ =\ \deg(u \cdot\x_F)^2.\qedhere \]
\end{proof}

\section{Open questions}\label{sec:discussion}

The non-lattice cases of Stanley's conjecture remain. Even for IDP lattice polytopes or stronger yet, for lattice polytopes with a regular unimodular triangulation, we are left with a gap in the inequalities restricting $h^{\ast}$ if the interior of the cone is generated in higher degree. We conjecture that the Lefschetz property, and in fact the unimodality of $h^{\ast}$ fails in general. The intuition here is that lattice polytopes behave like triangulated disks, which can have non-unimodal $h$-vectors. The idea here could rely on constructing appropriate connected sums: as we saw above, Gorenstein polytopes have $h^\ast$-polynomials that peak at half of their socle degree (which is $d+1-s$, $s$ being the minimal dilation constant so that the polytope has an interior vertex). By connected sum of polytopes with different socle degree, one could hope to turn a dromedary into a camel (though a mythical beast with more humps is not beyond our imagination, alas such a creature has to be high-dimensional). 

A word of caution, however, lies in an inequality for the $h^\ast$-polynomial arising from work of Eisenbud and Harris \cite{Stanleydomain}: we have that for any nonnegative $k$, and $s$ the degree of the $h^\ast$-polynomial, we have
\[h^\ast_0\ +\ \dots\ +\ h^\ast_k\ \le\ h^\ast_s\ +\ \dots\ +\ h^\ast_{s-k}.\]
This inequality is special to domains, and prevents us from introducing a hump below half the socle degree easily; it remains to understand the impact of this inequality in general. The most promising approach is then to look among polytopes whose interior is generated in high degree, and look for non-unimodality between half the degree of the $h^\ast$-polynomial and half the dimension of the polytope.

\textbf{Acknowledgements.} 
We thank Benjamin Nill for asking us about the local $h^\ast$ polynomial, and Paco Santos for insightful discussions.
K. A., V. P. and J. S. are supported by Horizon Europe ERC Grant number: 101045750 / Project acronym: HodgeGeoComb.

	{\small
		\bibliographystyle{myamsalpha}
		\bibliography{ref}}

\end{document}